\documentclass[a4paper]{article}
\usepackage{amsmath}
\usepackage{amssymb}
\usepackage{amsthm}
\usepackage{hyperref}
\usepackage{graphicx}
\usepackage{tikz}
\usepackage{mathtools}

\theoremstyle{plain}
\newtheorem{theorem}{Theorem}

\newtheorem*{theorem1*}{Theorem 1}

\newtheorem*{theorem2*}{Theorem 2}

\newtheorem{lemma}[theorem]{Lemma}

\theoremstyle{definition}
\newtheorem{definition}[theorem]{Definition}
\newtheorem{remark}[theorem]{Remark}
\newtheorem{example}[theorem]{Example}

\DeclareMathOperator{\Hom}{Hom}
\DeclareMathOperator{\Graph}{Graph}

\DeclareMathOperator{\id}{id}

\DeclareMathOperator{\Bord}{Bord}
\DeclareMathOperator{\ThreeBord}{Bord^\text{or}_{23}}
\DeclareMathOperator{\Vect}{Vect}
\DeclareMathOperator{\Null}{Null}
\DeclareMathOperator{\Surfaces}{Surfaces}

\newcommand{\be}{\begin{equation}}
\newcommand{\ee}{\end{equation}}
\newcommand{\ig}[1]{\begin{aligned} \includegraphics{#1} \end{aligned}}

\begin{document}

\title{Three-dimensional TQFTs via string-nets and two-dimensional surgery}
\author{Bruce Bartlett \\ \\
\small Mathematics Division, Stellenbosch University \\
\small Merriman Avenue, Stellenbosch \\
\small South Africa \\ 
\small and \\
\small National Institute for Theoretical and Computational Sciences (NITheCS) \\
\small South Africa \\
\small \tt bbartlett@sun.ac.za}
	
\maketitle

\begin{abstract}
If $C$ is a spherical fusion category, the string-net construction associates to each closed oriented surface $\Sigma$ the vector space $Z_\text{SN}(\Sigma)$ of linear combinations of $C$-labelled graphs on $\Sigma$ modulo local relations, in a way which is functorial with respect to orientation-preserving diffeomorphisms of surfaces. We show how to extend this assignment to a 3-dimensional topological quantum field theory (TQFT), by defining how the surgery generators in Juh\'{a}sz' presentation of the oriented 3-dimensional bordism category act on the string-net vector spaces. We show that the resulting TQFT, which is formulated completely in the two-dimensional graphical language of string-nets, is an alternative description of the Turaev-Viro state sum model.
\end{abstract}

\section{Introduction}
A 3-dimensional oriented topological quantum field theory (TQFT) is a symmetric monoidal functor \cite{segal, a88-tqft}
\be \label{defnTQFT}
 Z : \Bord^\text{or}_{23} \rightarrow \Vect
\ee
where $\Bord^\text{or}_{23}$ is the oriented 3-dimensional bordism category (objects are oriented closed surfaces and morphisms are 3-dimensional oriented cobordisms) and $\Vect$ is the category of vector spaces. Besides providing invariants of closed 3-manifolds and knots and links inside them, such a functor also provides a representation of the mapping class group $\Gamma(\Sigma)$ on the vector space $Z(\Sigma)$ assigned to a closed oriented surface $\Sigma$. Since the mapping class groups of surfaces are foundational in low-dimensional topology, it is precisely this property of a 3-dimensional TQFT which is often the most interesting. Some important results and applications in this regard are the asymptotic faithfulness of the representations coming from quantum groups \cite{andersen2006asymptotic, freedman2002quantum}, the fact that the image of these representations is infinite in general \cite{funar1999tqft},
the fact that mapping class groups do not have Kazhdan's property (T) \cite{andersen2007mapping}, and the asymptotic expansion conjecture for mapping tori \cite{andersen2019asymptotic}.

To construct a 3-dimensional oriented TQFT \eqref{defnTQFT} one needs some initial data, with very general initial data being that of a spherical fusion category $C$ \cite{bw99-sc}. The most well-known  way to construct this TQFT from $C$ directly (without first passing to the Drinfeld center of $C$) is the Turaev-Viro model \cite{turaev1992state}, a state-sum model which uses triangulations. See \cite[Chapter VII.3]{t94-qik} for the most complete description of this TQFT.

In the Turaev-Viro model, it is easy to compute the numerical invariant of a closed 3-manifold $M$ (pick a triangulation of $M$ and compute a state-sum). On the other hand, the vector space $Z(\Sigma)$ assigned to a closed surface $\Sigma$ is a bit more cumbersome since it is defined as a colimit over all triangulations of $\Sigma$. Hence, the drawback of the Turaev-Viro model is that the representation of the mapping class group on the vector space $Z(\Sigma)$, though explicit and well-defined, is difficult to work with in practice. 

In this paper we introduce an alternative approach to construct the 3-dimensional oriented TQFT arising from a spherical fusion category $C$. Our approach, which is geometric in nature and does not use triangulations or pants decompositions, combines the graphical calculus of {\em string-nets} (see \cite{kirillov2011string} for an overview) with the presentation of $\ThreeBord$ via {\em 2-dimensional surgery moves} due to Juh{\'a}sz \cite[Definition 1.4]{juhasz2018defining}. 

In our approach the vector space $Z(\Sigma)$ assigned to a closed oriented surface $\Sigma$ is simply the space of $C$-labelled string-nets on $\Sigma$, on which the mapping class group acts very easily and naturally. This part of our construction (the definition of the string-net vector space $Z(\Sigma)$ associated to an oriented surface $\Sigma$ and its functoriality with respect to oriented diffeomorphisms $\phi : \Sigma \rightarrow \Sigma'$) is not new. It originated in the physics literature with the work of Kitaev \cite{k05-aes} and Levin and Wen \cite{levin2005string} and a general mathematical account (treating carefully the case of surfaces with boundary) has been given by Kirillov \cite{kirillov2011string} (see also \cite{balsam2012turaev, goosen2018oriented} for useful overviews).

What {\em is} new in our approach is that we show how to extend the construction of these string-net vector spaces into a {\em full} TQFT by showing how to assign linear maps to Juh{\'a}sz's two-dimensional surgery moves in a way which satisfies the relations listed in \cite[Definition 1.4]{juhasz2018defining}. Our main results are the following.

\begin{theorem} Given a spherical fusion category $C$, the assignments $\Sigma \mapsto Z_\text{SN}(\Sigma)$ and $e \mapsto Z_\text{SN}(e)$ listed in Definition \ref{defn_of_string_modifications} satisfy Juh\'{a}sz' relations $\mathcal{R}$ and hence define a string-net TQFT $Z_\text{SN}$.
\end{theorem}

\begin{theorem} Given a spherical fusion category $C$, the string-net TQFT $Z_\text{SN}$ based on $C$ is naturally isomorphic to the Turaev-Viro TQFT $Z_\text{TV}$ based on $C$.
\end{theorem}
The string-net TQFT $Z_\text{SN}$ is an intrinsically two-dimensional way to define the TQFT which accords well with the intuitive idea that a spherical fusion category is an algebraic structure with an intrinsically two-dimensional graphical calculus. Indeed, this is the first motivation for this work --- to demonstrate that Juh{\'a}sz's presentation of $\ThreeBord$ fits together very neatly with the graphical calculus for spherical fusion categories. 

The second motivation is to try to make contact with recent work \cite{costantino2020kuperberg} extending the Turaev-Viro invariants of 3-manifolds to {\em non-semisimple} spherical tensor categories. The third motivation is to try and make contact with recent work on string-net models for {\em non-spherical} pivotal fusion categories \cite{runkel2020string}. We remark here for the experts that we use the spherical property in various ways in our construction. One interesting way it features is in the proof of conjugation invariance for surgery on a framed $0$-sphere (the Kirby loop must be invariant under orientation flip of the belt circle). 

Our approach is similar to that of Goosen \cite[Chapter 5]{goosen2018oriented} (see also \cite{bartlett2021extended}), who also used string-nets to construct a 3-dimensional oriented TQFT from a spherical fusion category $C$. However, whereas we use the presentation of the oriented bordism category $\Bord^\text{or}_{23}$ due to Juh{\'a}sz \cite{juhasz2018defining}, Goosen used the presentation of the oriented bordism {\em bi}category $\Bord^\text{or}_{123}$ due to the current author and collaborators \cite{PaperIV}. 

Juh{\'a}sz's presentation \cite[Definition 1.4]{juhasz2018defining} is infinite but very geometric since arbitrary surgery moves are allowed, while the presentation from \cite{PaperIV} is finite, less geometric and more combinatorial in nature since only a specific finite list of surgery moves is allowed. In practice, for instance, this means that working out the linear map $Z(M)$ associated to a 3-dimensional cobordism $M$ is easier in our approach here than in \cite{goosen2018oriented}. In our approach, one just needs a Morse function on $M$, while in Goosen's approach one needs to present $M$ explicitly as a composite of the generating 2-morphisms from \cite{PaperIV}.

This paper is organized as folllows. In Section \ref{string_net_sec} we review the construction of the string-net space of a closed oriented surface. In Section \ref{surgery_section} we review Juh\'{a}sz' two-dimensional surgery presentation of the three-dimensional oriented bordism category. In Section \ref{TQFT_section} we carefully define the cutting move, define linear maps associated to the surgery generators in Juh\'{a}sz' presentation, and check that that they satisfy the relations. Finally, we show that the resulting TQFT is naturally isomorphic to the Turaev-Viro model.

\subsection*{Acknowledgements} I would like to thank Andr\'{a}s Juh\'{a}sz, Gerrit Goosen, Nathan Geer and Azat Gainutdinov for helpful discussions and comments. This project has received funding from the European Research Council (ERC) under the European Union's Horizon 2020 research and innovation programme (grant agreement No 674978).

\section{String-nets} \label{string_net_sec}
In this section, we review the notion of the {\em string-net space} $Z_\text{SN}(\Sigma)$ of an oriented surface $\Sigma$, given the initial data of a spherical fusion category $C$. The main reference is \cite{kirillov2011string}, but we will essentially adopt the notation from Goosen's thesis \cite[Chapter 4]{goosen2018oriented}, which makes certain details explicit which are left implicit in \cite{kirillov2011string}. We will also add some details and observations of our own --- see Remarks \ref{how_to_evaluate}, \ref{frobenius_schur}, \ref{s_move_remark} and \ref{kirby_well_defined}.

\subsection{Graphical calculus for spherical fusion categories}
There is a well-known two-dimensional {\em string diagram} graphical calculus for working with monoidal categories (see \cite{selinger2010survey} for an overview, and \cite{bartlett2016fusion} for our conventions on associators). Our diagrams will go from top to bottom, so that a morphism $f : A \rightarrow B \otimes C$ is drawn as
\[
\ig{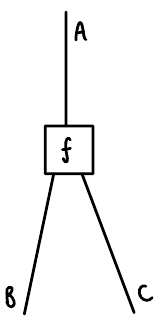} \, .
\]
More precisely, we say that the {\em value} of the above string diagram is $f$. We will refer to these planar diagrams, where there are clear top-to-bottom and left-to-right directions, where the coupons are rectangular, which depict a specific morphism in the monoidal category, and which are generally drawn in black, as {\em rectangular string diagrams}, to distinguish them from {\em string nets}, which live in a closed surface (see the next section), and which we will generally draw in red.

A {\em fusion category} $C$ is a rigid semisimple $\mathbb{C}$-linear monoidal category with finitely many isomorphism classes of simple objects and whose unit object is simple. This structure adds some new features to the graphical calculus. 

Firstly, note that since $1$ is simple, endomorphisms of $1$ can be canonically identified with complex numbers.

Secondly, rigidity means that every object $V$ now has a {\em dual} $V^*$. The edges in the graphical calculus are now given an orientation: a downward strand labelled $V$ refers to $V$ while an upward strand labelled $V$ refers to $V^*$. There exist right cap $\eta : 1 \rightarrow V^* \otimes V$ and cup $\epsilon : V \otimes V^* \rightarrow 1$ duality maps, drawn as
\[
 \ig{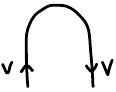}, \qquad \ig{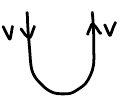} \, ,
\]
and left cap $n : 1 \rightarrow V \otimes V^*$ and cup $e : V^* \otimes V \rightarrow 1$ duality maps, drawn as
\[
 \ig{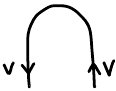}, \qquad \ig{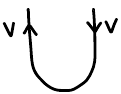} ,
\]
satisfying the rigidity equations
\[
 \ig{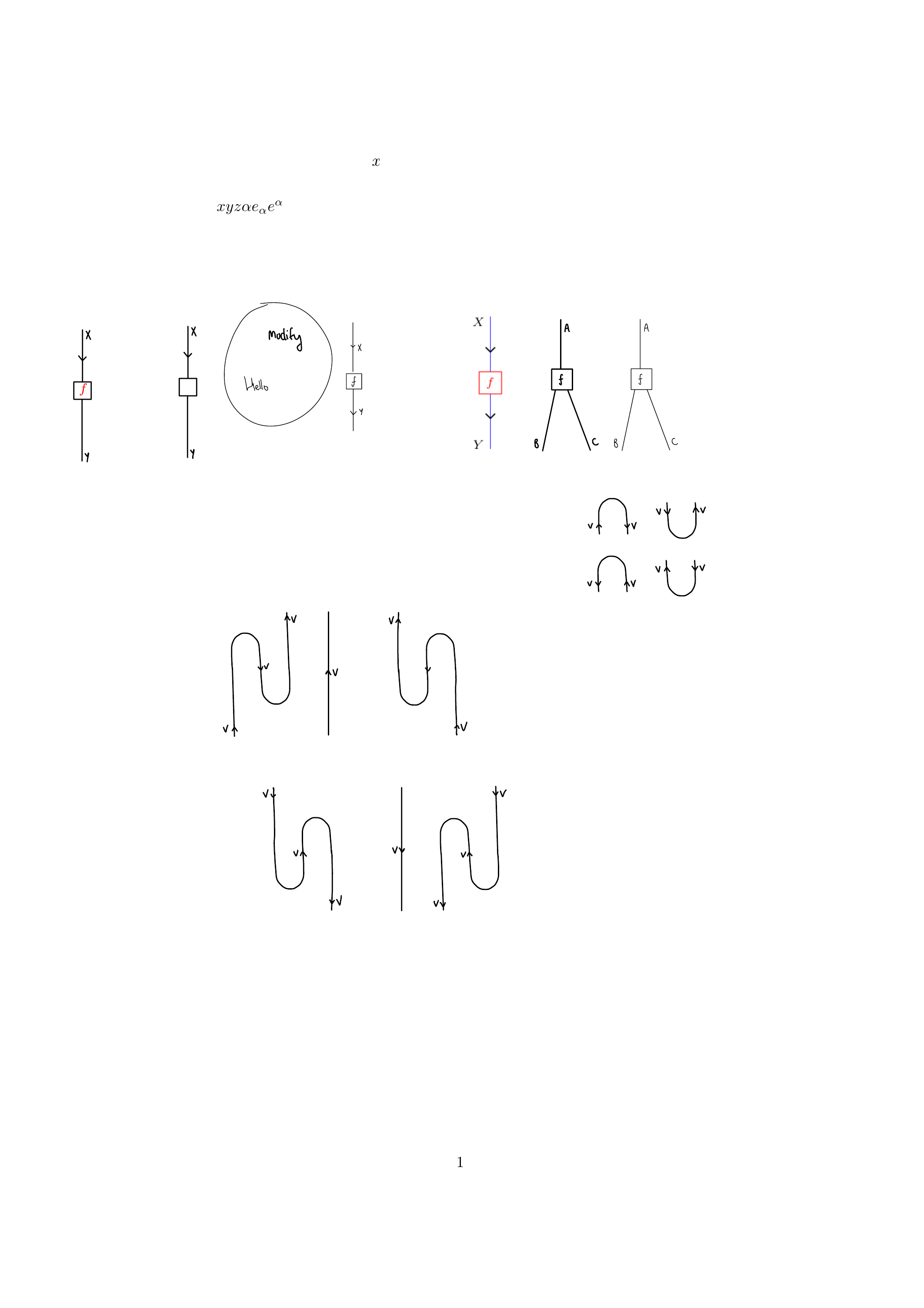} = \ig{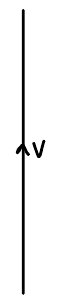} = \ig{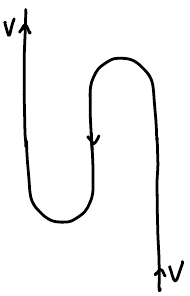} \quad , \quad \ig{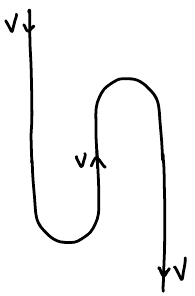}  = \ig{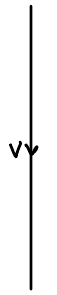}  = \ig{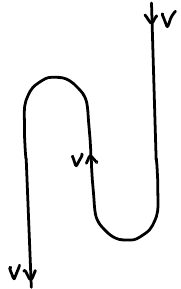} \, .
\]
Thirdly, semisimplicity means that there finitely many simple objects (i.e. their endomorphism vector space is 1-dimensional) $X_i$, $i = 1 \ldots n$, and for each pair of objects $V,W \in C$ the canonical composition map
\be \label{semisimple_eqn}
 \bigoplus_i \Hom(V, X_i) \otimes \Hom(X_i, W) \rightarrow \Hom(V, W)
\ee
is an isomorphism of vector spaces. In particular (by considering the summand where $X_i = 1$ and inserting cup and cap maps to interchange inputs to outputs), \eqref{semisimple_eqn} implies that for any objects $V_1, \ldots, V_n$ in $C$, the pairing
\begin{align}
 \Hom(1, V_1 \otimes \cdots \otimes V_n) \otimes \Hom(1, V_n^* \otimes \cdots \otimes V_1^*) & \rightarrow \mathbb{C} \label{pairing} \\
 \ig{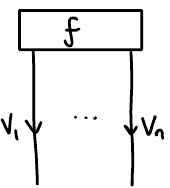} \otimes \ig{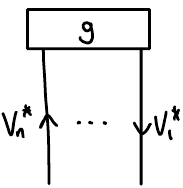} & \mapsto \ig{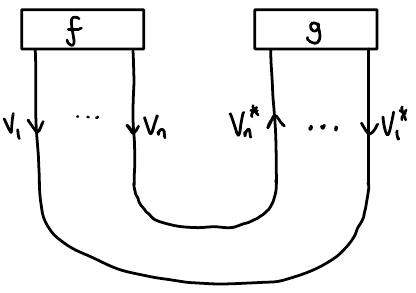}  
\end{align}
is nondegenerate.

A {\em pivotal structure} on a fusion category is a monoidal natural isomorphism $\gamma : \id \Rightarrow **$ where $*$ is the dualization functor. The key property this brings to the graphical calculus is that right and left duals agree, i.e.
\be \label{left_dual_equals_right_dual}
\ig{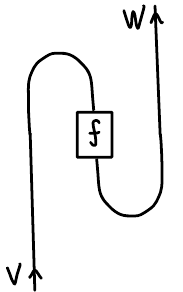} = \ig{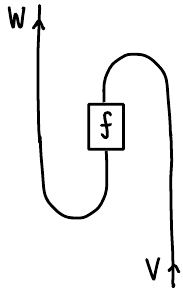}
\ee
for all morphisms $f : V \rightarrow W$. This means, in particular, that one can compute the pairing maps \eqref{pairing} using the right or left unit maps; the answer is the same:
\be \label{symmetry_of_pairing}
 \ig{c3.pdf} = \ig{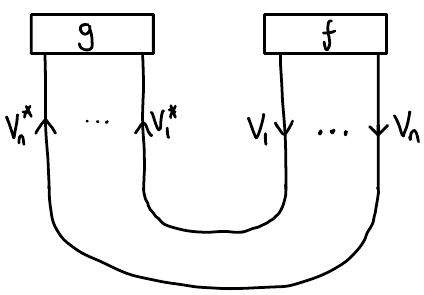}
\ee

The pivotal structure is called {\em spherical} if the two ways of computing the dimension of an object agree, i.e. if 
\[
 \ig{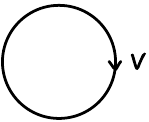} = \ig{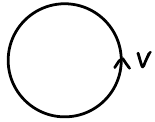}
\]
for all objects $V$. A {\em spherical fusion category} is a fusion category equipped with a spherical structure.

\subsection{$C$-labelled graphs}
For the purposes of this paper, a {\em graph} $G$ is a 1-dimensional CW-complex. We write $V(G)$ and $E(G)$ for the set of vertices and edges of $G$. Note that $V(G)$ and $E(G)$ are finite, and multiple edges and loops are allowed. An {\em oriented edge} $\mathbf{e}$ is a pair $(e, \text{or})$ where $e \in E(G)$ and $\text{or}$ is an choice of orientation for $e$. The set of oriented edges of $G$ is denoted $E^\text{or}(G)$. If $\mathbf{e}$ is an oriented edge, its underlying unoriented edge is written as $e$, while the oppositely oriented edge is written as $\overline{\mathbf{e}}$.

Let $\Sigma$ be a 2-dimensional closed smooth manifold. A {\em smoothly embedded graph in $\Sigma$} is a topological embedding
\[
 \phi : G \hookrightarrow \Sigma
\]
of a graph into $\Sigma$ whose restriction to the open 1-cells of $G$ is smooth. We will often identify the graph $G$ and its embedded image $\phi(G) \subset \Sigma$.  

Now let $C$ be a spherical fusion category.
\begin{definition} \label{defn_colored_graph} A {\em $C$-labelled graph in a closed oriented surface $\Sigma$} is a 4-tuple $(G, l, \epsilon, f)$, where:
\begin{itemize}
 \item $G$ is a smoothly embedded graph in $\Sigma$,
 \item $l$ is a map $E^\text{or}(G) \rightarrow \text{Ob}(C)$, satisfying $l(\overline{\mathbf{e}})$ = $l(\mathbf{e})^*$ for all $\mathbf{e} \in E^\text{or}(G)$,
 \item $\epsilon$ assigns to each vertex $v \in V(G)$ a choice of {\em initial half-edge} $\epsilon_v \in E(G)$ incident to $v$, and
 \item $f$ assigns to each vertex $v$ a morphism 
 \[
 f_v \in \Hom_C (1, l(\mathbf{e}_1)  \otimes \cdots \otimes l(\mathbf{e}_n)),
 \]
where the edges incident to $v$ are $\mathbf{e}_1, \ldots, \mathbf{e}_n$, taken in counterclockwise order according to the orientation of $\Sigma$, each with an outgoing orientation, and with $e_1 = \epsilon_v$.
\end{itemize}
We write $\Graph_C (\Sigma)$ for the collection of all $C$-labelled graphs in $\Sigma$.
\end{definition}
We will often abuse notation and refer to a $C$-labelled graph simply by its underlying graph $G$.

\subsection{Evaluation in a disk}
Let $\Sigma$ be an oriented surface, and $G$ a $C$-labelled graph in $\Sigma$. A {\em properly embedded disk in $\Sigma$} is a smooth orientation-preseving embedding
\[
  D \hookrightarrow \Sigma
\]
of the unit disk $D \subset \mathbb{R}^2$ into $\Sigma$, such that the edges of $G$ intersect the image of $\partial D$ transversely, and no edge intersects $p$, the image of $(1,0)$. We will often use the same symbol $D$ to refer to both the standard disk $D \subset \mathbb{R}^2$ and its embedded image in $\Sigma$.

\begin{definition} \label{eval_defn} Let $G$ be a $C$-labelled graph in $\Sigma$, and $D$ a smoothly embedded disk in $\Sigma$. Let $V_1, \ldots, V_n$ be the obects labelling the outgoing edges which intersect $\partial D$, taken in counterclockwise order starting from $p \in \partial D$.  The {\em evaluation of $G$ in $D$} is the morphism
\[
   \langle G \rangle_D \in \Hom_C (1, V_1 \otimes \cdots \otimes V_n)
\]
defined as the value of the rectangular string diagram obtained as follows:
\begin{itemize}
 \item Pull back $G \cap D$ to $\mathbb{R}^2$ via the embedding $D \hookrightarrow \Sigma$,
 \item Replace each vertex $v$ in $G$ with the rectangular coupon for $f_v$, and loop the edges round appropriately to mimic their appearance in a small neighborhood of $v$, 
 \item Loop the edges intersecting $\partial D$ round to the bottom of the diagram. 
\end{itemize}
\end{definition} 
\noindent See Figure \ref{evalfig} for an example.

\begin{figure}[t]
 \[
 \begin{aligned} \includegraphics[width=0.6\textwidth]{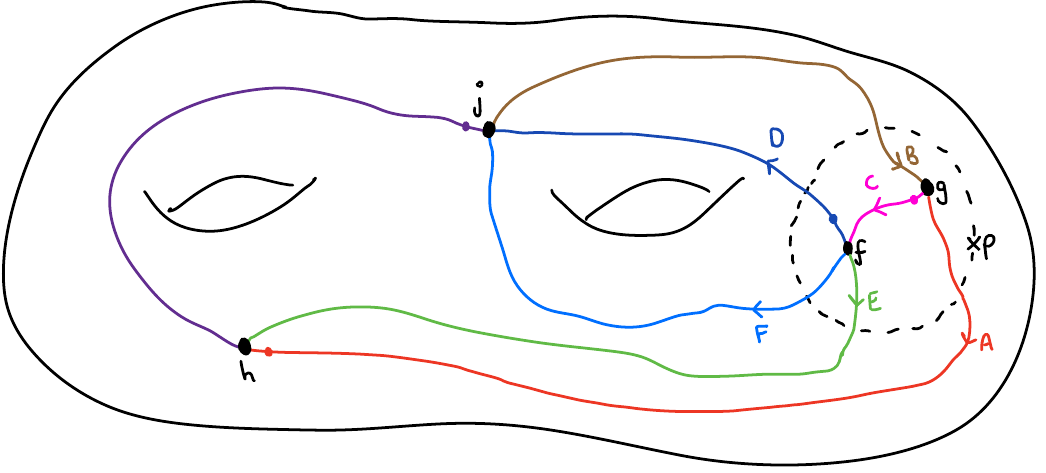} \end{aligned} \quad \begin{aligned} \includegraphics[width=0.4\textwidth]{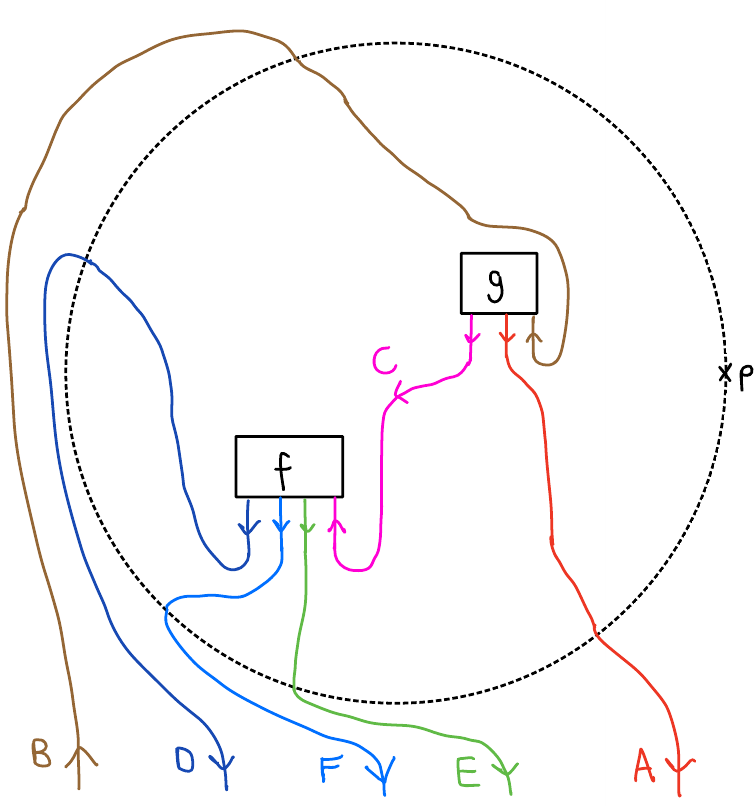} \end{aligned}
 \]
 \caption{\label{evalfig} The evaluation of the embedded disk is a rectangular string diagram representing a morphism $1 \rightarrow B^* \otimes D \otimes F \otimes E \otimes A$.}
\end{figure}

\begin{remark} \label{how_to_evaluate}There is a possible ambiguity here since we have not given a precise specification for the `loop the edges' instruction in the second and third steps of Definition \ref{eval_defn}. One could loop them clockwise or counterclockwise, possibly multiple times. However, one can verify that the pivotal structure on the category  --- in particular, equation \eqref{left_dual_equals_right_dual} --- implies that all ways of looping will evaluate to the same morphism.
\end{remark}

\begin{remark} Whenever we draw a $C$-labelled graph in $\mathbb{R}^2$ (such as the pullback of a graph in $\Sigma$ along an embedding $D \hookrightarrow \Sigma$), the orientation of the plane is understood to be counterclockwise. 
\end{remark}

\begin{lemma}\cite[Theorem 2.3]{kirillov2011string} The evaluation of a $C$-labelled graph in a disk $D \subset \Sigma$ has the following properties:
\begin{enumerate}
 \item If $G \cap D$ consists of a single vertex labelled by $f \in \Hom(1, V_1 \otimes \ldots V_n)$, then $\langle G \rangle_D = f$. That is,
 \[
 \Bigg \langle \ig{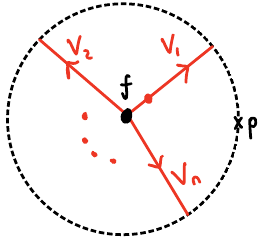} \Bigg \rangle_D = f.
 \]
 \item If $G \cap D$ and $G' \cap D$ differ only by isotopy, then $\langle G \rangle_D = \langle G' \rangle_D$.
 \item Rotating the choice of initial half-edge gives the equality
 \[
 \Bigg \langle \ig{b5.pdf} \Bigg \rangle_D = {\large \langle} \ig{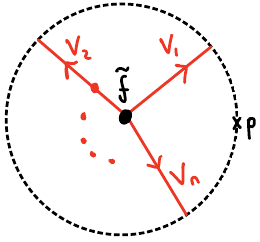} \Bigg \rangle_D 
 \]
 where
 \[
  \ig{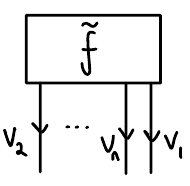} = \ig{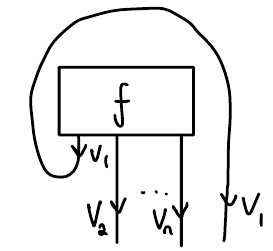} \, .
 \]
 \item The following `merging moves' hold:
 \begin{enumerate}
  \item For vertices:
  \[
   \Bigg \langle \ig{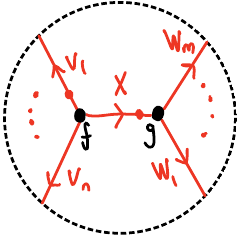} \Bigg \rangle_D = \Bigg \langle \ig{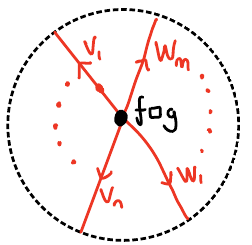} \Bigg \rangle_D
  \]
  where $f \square g$ is the morphism
  \[
   \ig{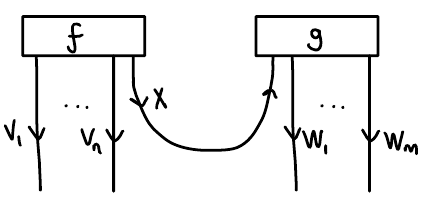} .
  \]
  \item For edges:
  \[
   \Bigg \langle \ig{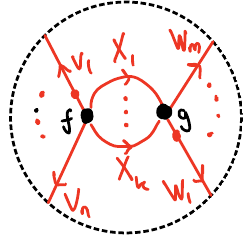} \Bigg \rangle_D = \Bigg \langle \ig{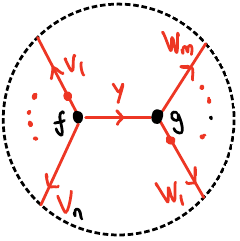} \Bigg \rangle_D
  \]
  where $Y=X_1 \otimes \cdots X_n$.
 \end{enumerate} 
 \item Suppose $f = a_1 f_1 + a_2 f_2$ holds inside $\Hom(1, V_1 \otimes \cdots \otimes V_n)$. Then
 \[
 \!\!\!\!\!\!\! \Bigg \langle \ig{b5.pdf} \Bigg \rangle_D = a_1 \Bigg \langle \ig{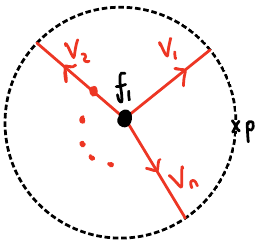} \Bigg \rangle_D + \,\, a_2 \Bigg \langle \ig{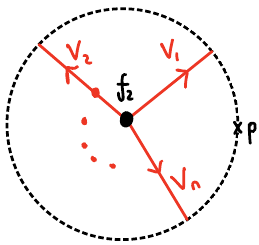} \Bigg \rangle_D \, .
 \]
\end{enumerate}
\end{lemma}
We also have the following graphical expression of the semisimplicity isomorphism \eqref{semisimple_eqn}.
\begin{lemma}\cite[Theorem 3.4 (5)]{kirillov2011string} \label{resolve_identity} For any object $V$,
\[
\Bigg \langle \ig{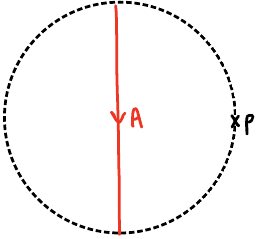} \Bigg \rangle_D = \sum_{i, \alpha}  d_i \, \, \Bigg \langle \ig{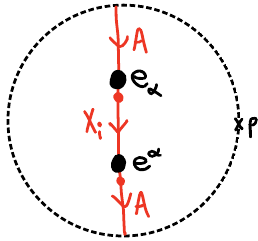} \Bigg \rangle \, 
\]
where $e_\alpha$ is a basis for $\Hom(1, X_i \otimes A^*)$ and $e^\alpha$ is the dual basis for $\Hom(1, A^* \otimes X_i)$ according to the pairing \eqref{pairing}.
\end{lemma} 
\begin{remark} \label{frobenius_schur} The way that labellings of oriented edges are defined in Definition \ref{defn_colored_graph}, and that evaluation is defined in Definiton \ref{eval_defn}, bears careful thought. The reader will verify that for a self-dual simple object $X$, we have
\[
 \Bigg \langle \ig{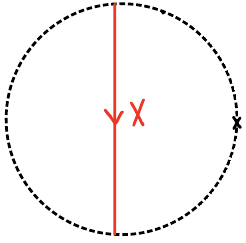} \Bigg \rangle_D = \nu_{X} \Bigg \langle  \ig{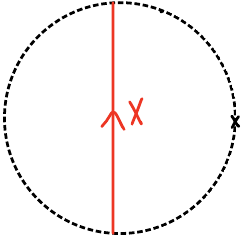} \Bigg \rangle_D
\]
where $\nu_X = \pm 1$ is the Frobenius-Schur indicator of $X$. This is because the left-hand side evaluates to the counit $\eta : 1 \rightarrow X^* \otimes X$ 
\[
  \ig{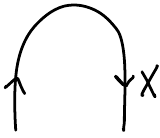}
\]
expressing $X^*$ as a right dual of $X$, while the graph sinde the disk on the right-hand side evaluates to the counit $n : 1 \rightarrow X \otimes X^*$ 
\[
 \ig{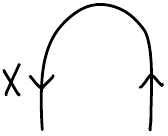}
\]
expressing $X^*$ as a left dual of $X$. For a self-dual object $X$, $n = \nu_X \eta$.
\end{remark}

\subsection{The string-net space}
Let $\Sigma$ be an oriented surface. Let $\mathbb{C}[\Graph_C(\Sigma)]$ be the vector space spanned by $C$-labelled graphs in $\Sigma$. 

\begin{definition} Let $D$ be an embedded disk in $\Sigma$. A {\em null relation relative to $D$} is a formal linear combination
\[
 a_1 G_1 + \cdots + a_n G_n \in \mathbb{C}[\Graph_C(\Sigma)]
\]
where $G_1, \ldots, G_n$ are identical on the complement of $D$, such that
\[
  a_1 \langle G_1 \rangle_D + \cdots + a_n \langle G_n \rangle_D = 0 \,.
\]
We write $\Null_C (\Sigma)$ for the subspace of $\mathbb{C}[\Graph_C(\Sigma)]$ formed by all null relations, relative to all possible embedded disks $D \hookrightarrow \Sigma$.
\end{definition}

\begin{definition} Let $\Sigma$ be a closed oriented smooth surface, and $C$ a spherical fusion category. We define the {\em $C$-labelled string-net space of $\Sigma$} as
\[
 Z_\text{SN}(\Sigma) := \mathbb{C}[\Graph_C(\Sigma)] / \Null_C (\Sigma)
\]
Elements of $Z_\text{SN}(\Sigma)$ are called {\em string-nets}. The equivalence class of $G$ is written as $\langle G \rangle$.
\end{definition}
In other words, two $C$-labelled graphs in $\Sigma$ are equivalent as string-nets if one can be transformed into the other by a finite sequence of local relations holding in disks. The following are a natural consequences of the definitions.

\begin{lemma} Let $G, G' \in \Graph_C(\Sigma)$. If $G$ is isotopic to $G'$, then $\langle G \rangle = \langle G' \rangle$.
\end{lemma}

\begin{example} \label{string_net_S2} There is a canonical isomorphism $Z_\text{SN} (S^2) \rightarrow \mathbb{C}$, obtained by isotoping the entire string-net into a disk and then evaluating it. 
\end{example}

Given an orientation-preserving diffeomorphism $\phi : \Sigma \rightarrow \Sigma'$, there is a natural push-forward map
\[
 Z_\text{SN}(\phi) : Z_\text{SN}(\Sigma) \rightarrow Z_\text{SN}(\Sigma')
\]
defined by sending string-nets in $\Sigma$ to their image in $\Sigma'$. 
\begin{remark} \label{s_move_remark} Making this push-forward as explicit as possible is the reason we have chosen to equip the vertices of our $C$-labelled graphs with explicitly chosen initial half-edges (in contrast to \cite{kirillov2011string}, where cyclic reorderings are implicitly identified by canonical isomorphisms). It is instructive to consider the case where $X$ is a self-dual object, and $f \in \Hom(1, X \otimes X \otimes X \otimes X)$. Consider the following string-net on a torus:
\[
 \ig{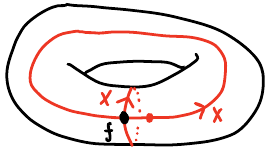}
\]
Firstly, since $X$ is self-dual, the orientations on the strands do not matter, and we redraw the string-net simply as:
\be \label{before_s}
 \ig{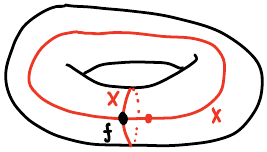}
\ee
Now, let $s : \Sigma \rightarrow \Sigma$ be the clockwise `s-move', given by rotating $\mathbb{R}^2$ clockwise by 90 degrees in the presentation of the torus as $\mathbb{R}^2 / \mathbb{Z}^2$. Then the push-forward map $Z_\text{SN}(s)$ sends
\be \label{after_s}
 \ig{f6.pdf} \mapsto \ig{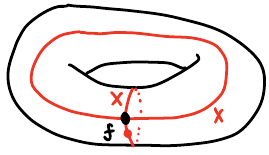} \, .
\ee
In order to be able to compare the right-hand side of \eqref{after_s} with \eqref{before_s}, we must first rotate the chosen initial half-edge of $\eqref{after_s}$ counterclockwise:
\[
 \ig{f11.pdf} = \ig{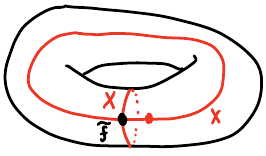}
\]
We see that $Z_\text{SN}(s)$ has the effect of sending $f \mapsto \tilde{f}$,
\[
 \ig{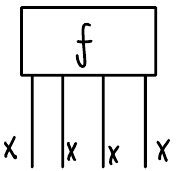} \mapsto \ig{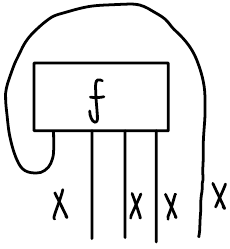} \, .
\]
In other words, $Z_\text{SN}(s)$ is precisely Ng and Schaunberg's map $E^{(4)}_X$, whose trace gives the 4th higher Frobenius-Schur indicator of $X$. If we had not chosen initial half-edges at vertices, this action would be harder to see.
\end{remark}

In summary, we have:

\begin{theorem} The string-net space construction is a monoidal functor
\[
 Z_\text{SN} : \Surfaces \rightarrow \Vect
\]
where $\Surfaces$ is the category whose objects are closed oriented surfaces, whose morphisms are orientation-preserving diffeomorphisms, and where the tensor product is disjoint union.
\end{theorem}

\subsection{Cloaking}
The last ingredient we need in the graphical calculus are `Kirby loops'. 

\begin{definition}[Kirby loops]   Let an unlabelled orange edge in a $C$-labelled graph $G$ be shorthand for the string-net defined as a sum of copies of $G$, where the unlabelled edge has been labelled by the simple objects $X_i$, each graph weighted by the dimension $d_i$ of $X_i$:
\[
   \ig{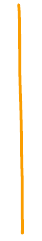} := \sum_i d_i \ig{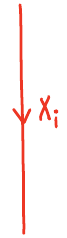}
\]
\end{definition}
\begin{remark} 
\label{kirby_well_defined} Note that this is well-defined (without specifying an orientation on the orange strand) since $d_i = d_{i^*}$.
\end{remark}

\begin{lemma}[Cloaking Lemma] \label{cloaking_lemma} \cite[Corollary 3.5]{kirillov2011string})
In the string-net space, the following relation holds in any annular region of the surface, for any object $V$:
\[
 \ig{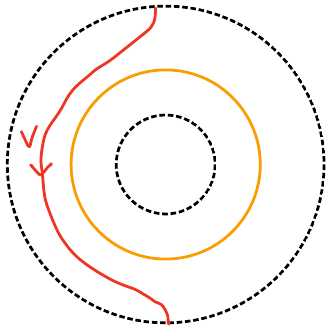} = \ig{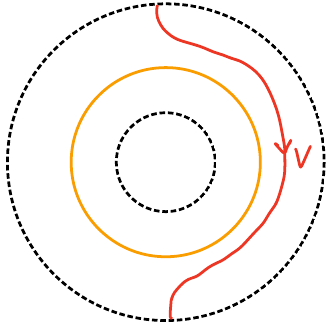}
\]
\end{lemma}

\section{The surgery presentation of Juh{\'a}sz}
\label{surgery_section}
In this section we review the surgery presentation of $\Bord_{23}^\text{or}$ due to Juh\'{a}sz \cite[Definition 1.4]{juhasz2018defining}. 

\begin{remark}
In \cite{juhasz2018defining}, Juh\'{a}sz actually gives {\em two} methods of defining a 3-dimensional TQFT --- either by assigning linear maps to the surgery presentation from \cite[Definition 1.4]{juhasz2018defining} which satisfy the relations, or else by constructing a {\em $J$-algebra}. In this paper, we are concerned with the former approach.
\end{remark}

\subsection{Morse theory approach to TQFT}
It has been recognized from the beginning that Morse theory is a useful approach to TQFT \cite{kontsevich1988rational, walker-notes} . A Morse function $f : M \rightarrow [0,1]$ breaks up a cobordism $M : \Sigma_1 \rightarrow \Sigma_2$ into time slices $\Sigma_t = f^{-1}(t)$, and so assigning a linear map to the cobordism boils down to analysing the different kinds of topology change which can occur to $\Sigma_t$ when $t$ passes through a critical value, and assigning a linear map to each of them. If $n$ is the dimension of the time slice $\Sigma_t$ (so that the TQFT is $(n+1)$-dimensional), then Morse theory tells us (for an overview, see \cite{\cite{mil65-lhcob}}) that if there is a single critical point between $t_1$ and $t_2$, then $\Sigma_{t_1}$ will differ from $\Sigma_{t_2}$ by {\em surgery on a $k$-sphere}, for $k=-1, \ldots, n$ (for $k=-1$ this is defined as birth of an $n$-dimensional sphere). On the other hand, if there is no critical point between $t_1$ and $t_2$, then $\Sigma_{t_1}$ differs from $\Sigma_{t_2}$ by a diffeomorphism. 

Carrying out this programme precisely enough to obtain a presentation for the bordism category $\Bord_{n, n+1}^\text{or}$ is not trivial, since one must keep careful track of the attaching spheres for the surgeries. This was done by Juh\'{a}sz for all dimensions $n$ in \cite{juhasz2018defining}, although in this paper, we are only concerned with the case $n=2$.

\subsection{Surgery along framed spheres}

We set $S^{-1} = \phi$. Let $\Sigma$ be an oriented 2-dimensional smooth surface. Topologically, surgery on an embedded $k$-sphere in $\Sigma$ removes a thickened $k$-sphere from $\Sigma$ and glues back in a thickened $S^{1-k}$-sphere. An elegant way to perform this surgery in a smooth way was given by Milnor.
\begin{definition} \cite[Defn 3.11]{mil65-lhcob} Let $\Sigma$ be an oriented surface. For $k \in \{0, 1, 2\}$, a {\em framed $k$-sphere in $\Sigma$} is an orientation-reversing embedding
\[
 S : S^k \times OD^{2-k} \hookrightarrow \Sigma
\]
where $OD^p$ is the open unit disk of dimension $p$. Then {\em $\Sigma$ surgered along $S$} is the quotient smooth manifold
\[
 \Sigma(S) := \left(\Sigma \setminus S(S^{k} \times 0)\right) \sqcup (OD^{k+1} \times S^{1-k})
\]
obtained by identifying $S(u, \theta v)$ with $(\theta u, v)$ for all $u \in S^k, v \in S^{1-k}$, $0 < \theta < 1$. The {\em attaching sphere} of the surgery $S$ is 
\[
a_S := S(S^k \times 0) \subset \Sigma
\]
and its {\em belt sphere} is 
\[
 b_S := 0 \times S^{1-k} \subset \Sigma(S) .
 \] 
\end{definition}
See Figure \ref{surg_fig}. 
\begin{figure}[t]
\begin{align*}
 \ig{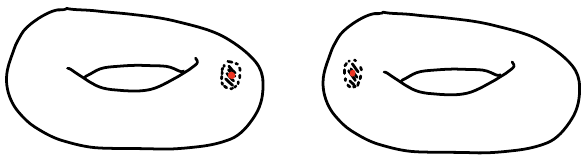} & \rightarrow \ig{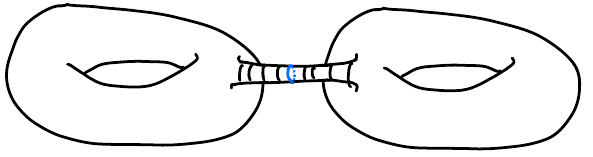} \\
 \ig{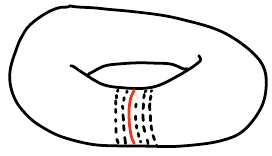} & \rightarrow \ig{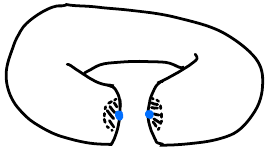}  
\end{align*}
\caption{\label{surg_fig} Surgery on a surface along a framed $0$-sphere and a framed $1$-sphere respectively. The attaching spheres of the surgeries are shown in red, while the belt spheres are shown in blue.} 
\end{figure} 

Note that away from the attaching and belt spheres we have a canonical diffeomorphism,
\[
  \phi_S : S \setminus a_S \rightarrow \Sigma(S) \setminus b_S .
\] 
For a surgery $S$ on a $0$- or a $1$-sphere, we define its {\em conjugate} $\overline{S}$ by:
\begin{align*}
 \overline{S}(x_1, y_1, y_2) &= S(-x_1, -y_1, y_2)   && (\text{$0$-sphere})\\
 \overline{S}(x_1, x_2, y_1) &= S(-x_1, x_2, -y_1)   && (\text{$1$-sphere})
\end{align*}

If $S$ is a framed $k$-sphere in $\Sigma$, then one can canonically construct a smooth cobordism $M : \Sigma \rightarrow \Sigma(S)$ (we adopt Milnor's construction from \cite[Theorem 3.12]{mil65-lhcob}) called the {\em trace of $S$}. This cobordism $M_S$ carries a canonical Morse function and gradient-like vector field, whose  associated surgery sphere is $S$.

\subsection{The presentation}
We can now write down Juh\'{a}sz' surgery presentation \cite[Definition 1.4]{juhasz2018defining} for $\Bord_{23}^\text{or}$. 
\begin{remark} In fact, Juh\'{a}sz gave a presentation for $\Bord_\text{n, n+1}^\text{or}$ for general $n$, but we only need the case $n=2$, which allows us to make some simplifications. In particular the critical-point cancellation diffeomorphism $\varphi$ in relation R5 is much easier to express. On the other hand, we do add some details not present in \cite{juhasz2018defining}, such as  being more precise in the notation for disjoint surgery spheres (R4). \end{remark}

\begin{definition} \label{juhasz_generators}
Let $\mathcal{G}$ be the directed graph where a vertex is an oriented smooth manifold $\Sigma$, and the edges are given by:
\begin{itemize}
 \item $\Sigma \stackrel{e_\phi}{\longrightarrow} \Sigma'$, where $\phi : \Sigma \rightarrow \Sigma'$ is an orientation-preserving diffemorphism,
 \item $\Sigma \stackrel{e_S}{\longrightarrow} S(\Sigma)$, where $S$ is a framed sphere in $\Sigma$.
\end{itemize}
\end{definition} 

\begin{definition} We define the following set $\mathcal{R}$ of relations in the free category $\mathcal{F}(\mathcal{G})$:
\begin{enumerate}
\item[R1.] (Isotopic diffeomorphisms) If $\phi$ is isotopic to $\phi'$ then $e_\phi \sim e_{\phi'}$.
\item[R2.] (Composition of diffeomorphisms) $e_{\phi' \circ \phi} \sim e_{\phi'} \circ e_{\phi}$ for composable diffeomorphisms $\phi'$ and $\phi$.
\item[R3.] (Surgery-diffeomorphism naturality) If $\phi : \Sigma \rightarrow \Sigma'$ is an orientation preserving diffeomorphism and $S \subset \Sigma$ is a framed sphere, let $S'= \phi \circ S$ and let $\phi^{S} : \Sigma(S) \rightarrow \Sigma'(S')$ be the induced diffeomorphism. Then the following commutative diagram is a relation:
\[
 \begin{tikzpicture}[scale=1.6]
 	\node (1) at (0,0) {$\Sigma$};
	\node (2) at (1,0) {$\Sigma(S)$};
	\node (3) at (0,-1) {$\Sigma'$};
	\node (4) at (1,-1) {$\Sigma'(S')$};
	\draw[->] (1) to node[above]{$e_{S}$} (2);
	\draw[->] (1) to node[left]{$e_{\phi}$} (3);
	\draw[->] (2) to node[right]{$e_{\phi^S}$} (4);
	\draw[->] (3) to node[below] {$e_{S'}$} (4);
 \end{tikzpicture}
\]
\item[R4.] (Disjoint surgeries commute) If $S$ and $T$ are disjoint framed spheres in $\Sigma$, let $S' = \phi_T \circ S$ and $T' = \phi_S \circ T$, and let 
\[
 \phi : \Sigma(S)(T') \rightarrow \Sigma(T)(S')
\]
be the canonical diffeomorphism. Then the following commutative diagram is a relation:
\[
 \begin{tikzpicture}[xscale=3,yscale=1.6]
 	\node(1) at (0,0) {$\Sigma$};
	\node (2) at (2,0) {$\Sigma(T)$};
	\node (3) at (0, -1) {$\Sigma(S)$};
	\node (4) at (1, -1) {$\Sigma(S)(T')$};
	\node (5) at (2, -1) {$\Sigma(T)(S')$};
	
	\draw[->] (1) to node[above]{$e_T$} (2);
	\draw[->] (2) to node[right]{$e_{S'}$} (5);
	\draw[->] (1) to node[left]{$e_S$} (3);
	\draw[->] (3) to node[below]{$e_{T'}$} (4);
	\draw[->] (4) to node[below]{$e_{\phi}$} (5);
 \end{tikzpicture}
\]
\item[R5.] (Critical point cancellation) If $S$ is a framed $k$-sphere in $\Sigma$ and $S'$ is a framed $(k+1)$-sphere in $\Sigma(S)$ such that the belt sphere $b_S$ of $S$ intersects the attaching sphere $a_{S'}$ of $S'$ once (this can only happen for $k=-1$ or $k=0$), then there exists a unique (up to isotopy) diffeomorphism $\varphi : \Sigma \rightarrow \Sigma(S)(S')$ which is the identity on $\Sigma \cap \Sigma(S)(S')$. Then we have the following relation (see Figure \ref{cancellation_fig}):
\[
  \begin{tikzpicture}[scale=1.6]
    \node (1) at (0,0) {$\Sigma$};
    \node (2) at (1,0) {$\Sigma(S)$};
    \node (3) at (1,-1) {$\Sigma(S)(S')$};
    
    \draw[->] (1) to node[above]{$e_S$} (2);
    \draw[->] (2) to node[right]{$e_{S'}$} (3);
    \draw[->] (1) to node[below left]{$e_\varphi$} (3);
  \end{tikzpicture}
\]

\item[R6.] (Conjugation invariance) For surgeries $S$ on $0$- and $1$-spheres, $e_S \sim e_{\overline{S}}$. 
\end{enumerate}
\end{definition}
\begin{figure}[t]
\[
\ig{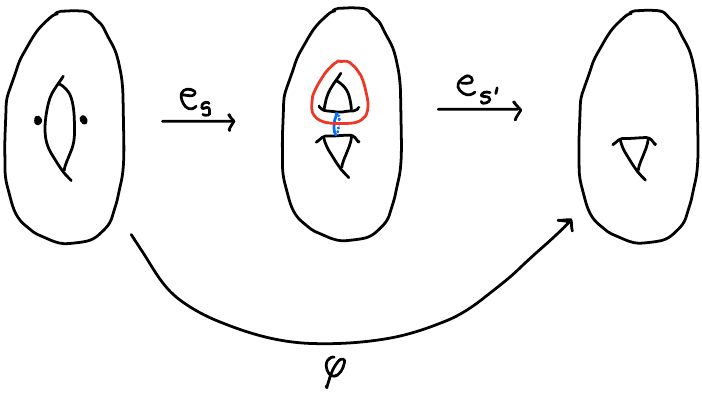}
\]
\caption{\label{cancellation_fig}Critical point cancellation for surgery on a $0$-sphere followed by surgery on a $1$-sphere. The belt sphere $a_S$ is shown in red while the attaching sphere $a_{S'}$ is shown in blue.}
\end{figure}
Having defined the relations, we obtain the quotient category $\mathcal{F}(\mathcal{G})/\mathcal{R}$. It has a natural symmetric monoidal structure when equipped with the disjoint union operation.

\begin{theorem}\cite[Theorem 1.7]{juhasz2018defining} The graph morphism
\[
 c : \mathcal{G} \rightarrow \ThreeBord
\]
which is the identity on objects, which sends a diffeomorphism edge $e_\phi$ to its corresponding mapping cylinder cobordism $M_\phi$, and which sends a surgery edge $e_S$ to is associated trace cobordism $M_S$, descends to a functor
\[
 c : F(\mathcal{G}) / \mathcal{R} \rightarrow \ThreeBord
\]
which is an equivalence of symmetric monoidal categories.
\end{theorem}

\section{String-nets as a TQFT} \label{TQFT_section}
In this section we upgrade the string-net functor
\[
 Z_\text{SN} : \Surfaces \rightarrow \Vect
\]
from Section \ref{string_net_sec} into a full 3-dimensional TQFT
\[
 Z_\text{SN} : \ThreeBord \rightarrow \Vect
\]
by defining linear maps associated to Juh\'{a}sz' surgery moves.

\subsection{The cutting move}
First we will need to define the cutting move.
\begin{definition}\label{cutting_defn}(Cutting move) Let $G$ be a $C$-labelled graph in $\Sigma$. Suppose that in some embedded disk $D \subset \Sigma$, 
\be \label{good_position}
 G \cap D = \ig{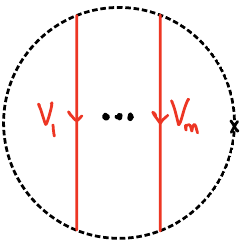} \, .
\ee
Then {\em cutting the strands inside $D$} means replacing $G$ in $D$ as follows,
\[
 \ig{d6.pdf} \quad \mapsto \quad \ig{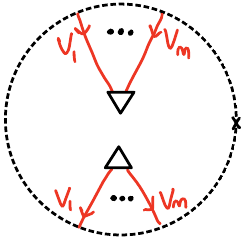} := \sum_\alpha \ig{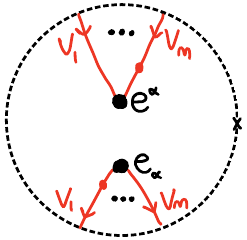} \, ,
\]
where $e_\alpha$ is a basis for  $\Hom( 1,V_1 \otimes V_m)$ and $e^\alpha \in \Hom(1, V_m^* \otimes \cdots \otimes V_1^*)$ is its dual basis, according to the pairing \eqref{pairing}.
\end{definition}
\begin{lemma} (Invariance of cutting) \label{invariance_of_cutting_lemma}Suppose that
\[
 \ig{d6.pdf} \sim \ig{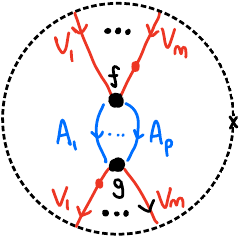}
\]
in the sense that their evaluations are equal. Then
\[
 \ig{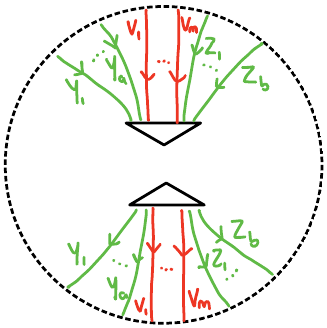} \sim \ig{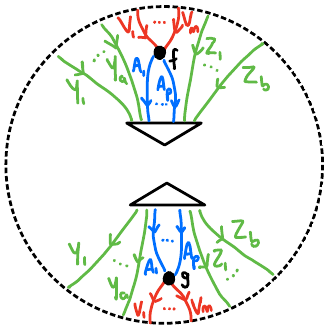} \, .
\]
\end{lemma}
\begin{proof}
This follows from semisimplicity, by inserting the resolution of the identity on each strand as in Lemma \ref{resolve_identity}.
\end{proof}

\subsection{The string-net TQFT}
We can now define the string-net TQFT.
\begin{definition} \label{defn_of_string_modifications} We define linear maps $Z_\text{SN}(e)$ associated to the generating edges $e$ of Juh\'{a}sz' surgery presentation $\mathcal{G}$ by modifying the string-nets as follows:
\begin{itemize}
 \item (Diffeomorphism) If $\phi : \Sigma \rightarrow \Sigma'$ is an orientation-preserving diffeomorphism, define $Z_\text{SN}(e_\phi) : Z_\text{SN}(\Sigma) \rightarrow Z_\text{SN}(\Sigma')$ by pushing forward string-nets along $\phi$.
 \item (Surgery on a -1-sphere) If $S$ is birth of a 2-sphere, then define $Z_\text{SN}(e_S)$ simply by putting the empty string-net on the newly created $S^2$, and multiply by $\frac{1}{D^2}$. In pictures:
\[
 \ig{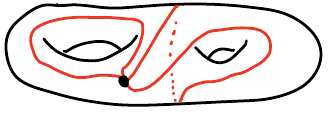} \mapsto \frac{1}{D^2} \ig{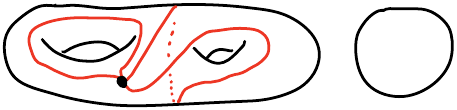}
\] 
 \item (Surgery on a 0-sphere) If $S$ is surgery on a framed $0$-sphere, define $Z_\text{SN}(e_S)$ by first isotoping any strands in the string-net away from the two surgery disks, then perform the surgery, and then insert a Kirby loop into the string-net around the belt-sphere $b_S$ of $S$:
 \[
 \ig{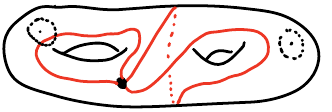} \mapsto \ig{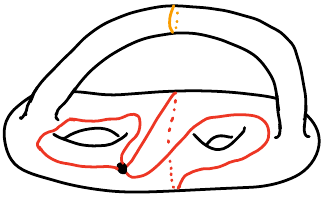}
 \]
 \item (Surgery on a 1-sphere) Suppose $S :S^1 \times OD^1 \hookrightarrow \Sigma$ is a framed $1$-sphere in $\Sigma$. Let $B_{\delta}$ be the complement of a small $\delta$-neighborhood of $p=(1,0) \in S^1$, and let $C_\epsilon$ be the complement of a small $\epsilon$-neighborhood of $0 \in OD^1$. By choosing $\delta$ and $\epsilon$ small enough, we can ensure that the pullback along $S$ of the portion of the string-net inside the image of the disk $D = B \times C$ (which lives in a small neighborhood of the attaching sphere $a_S$ of $S$) consists of $m$ parallel strands as in \eqref{cutting_defn}. Then define $Z_\text{SN}(e_S)$ by cutting the strands inside $D$, and performing the surgery. In pictures:
 \[
 \ig{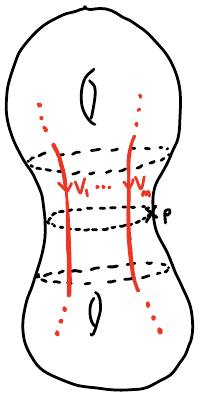} \mapsto \ig{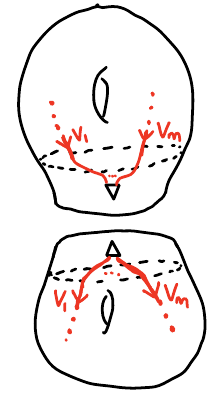}
 \]
 \item (Surgery on a 2-sphere) If $S$ is death of a 2-sphere inside $\Sigma$, then define $Z_\text{SN}(e_S)$ simply by evaluating the string-net on the 2-sphere to get a number (see Example \ref{string_net_S2}), then removing the copy of $S^2$ and multiplying the remaining string-net on $\Sigma \setminus S^2$ by this number. In pictures,
 \[
  \ig{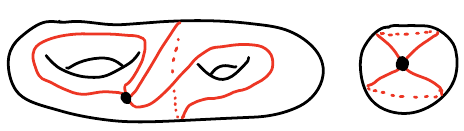} \mapsto \lambda \ig{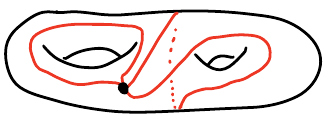}
 \]
where $\lambda$ is the evaluation of the string-net on the copy of $S^2$. 
 
\end{itemize}

\begin{lemma} \label{well_defined} The above assignments are well-defined, i.e. they do only depend on the string-net equivalence class of the $C$-labelled graphs used to compute them.
\end{lemma}
\begin{proof} The two assignments we need to check are surgery on a $0$- and $1-$sphere respectively.

For surgery on a $0$-sphere, we need to ensure that the initial preparatory step of isotoping the strands away from the two surgery disks does not lead to different final results after the surgery. Happily, this is ensured by cloaking (Lemma \ref{cloaking_lemma}).

Now consider surgery on a framed $1$-sphere $S$. Let $G$ be a $C$-labelled graph $G$ in $\Sigma$. We need to check that cutting $G$ along the attaching-sphere $a_S$ of $S$ in the manner outlined in Definition \ref{defn_of_string_modifications} does not depend on any possible modification to $G$ made by applying a local relation in a disk $D$. 

If $D$ does not intersect $a_S$, then this is clear, since the same modification can be made before or after the cutting.

If $D$ intersects $a_S$, then this follows from Lemma \ref{invariance_of_cutting_lemma}.

\end{proof} 

We can now prove our main theorem.
\begin{theorem1*} Given a spherical fusion category $C$, the assignments $\Sigma \mapsto Z_\text{SN}(\Sigma)$ and $e \mapsto Z_\text{SN}(e)$ listed in Definition \ref{defn_of_string_modifications} satisfy Juh\'{a}sz' relations $\mathcal{R}$ and hence define a string-net TQFT $Z_\text{SN}$.
\end{theorem1*}
\begin{proof}
\underline{(R1 to R4)} These are clearly satisfied.

\underline{R5 (case $k=0$)}. Application of $Z_\text{SN}(e_S)$ and then $Z_\text{SN}(e_{S'})$ clearly leaves the string-net unchanged, since the Kirby loop lives in a disk and evaluates to $D^2$. In pictures:
\begin{align*}
 \ig{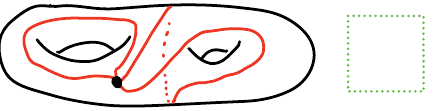} & \xmapsto{Z_\text{SN}(e_S)} \frac{1}{D^2} \ig{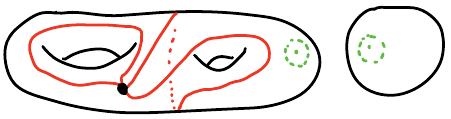} \\
 & \xmapsto{Z_\text{SN}(e_{S'})} \frac{1}{D^2} \ig{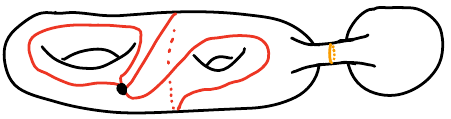} \\ 
 & = \ig{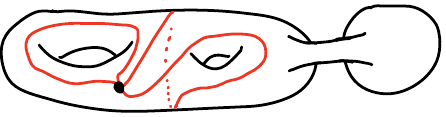}
\end{align*}
The final string-net is precisely $Z_\text{SN}(e_{\varphi})$ applied to the initial string-net, as it should be.

\underline{R5 (case $k=1$)}. Let $c = \phi_{S}^{-1}(a_{S'})$ be the pull-back to $\Sigma$ of the relevant segment of the attaching circle $a_{S'} \subset \Sigma(S)$ along the diffeomorphism $\phi_S : \Sigma \setminus a_S \rightarrow \Sigma' \setminus b_S$. Then $c$ is a curve in $\Sigma$ connecting the two points of the attaching $0$-sphere $a_S$:
\[
\ig{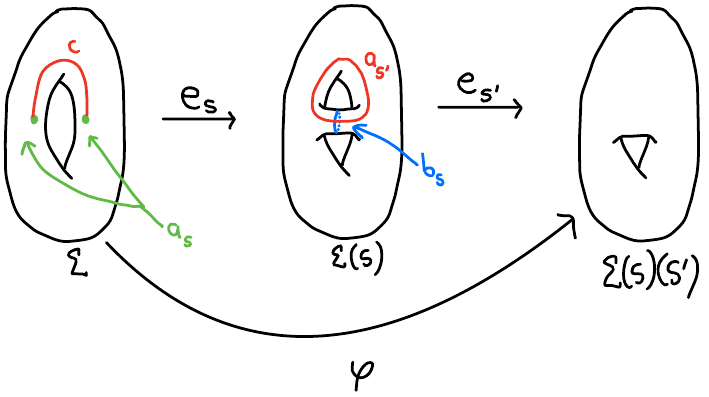}
\]
To check that the diagram commutes when we apply $Z_\text{SN}$ to the edges, we first use Lemma \ref{well_defined} to isotope the string-net in $\Sigma$ so that it does not intersect $c$. Then when we apply the cutting move $Z_\text{SN}(e_{S'})$ we only have to cut the Kirby loop, which is living on the belt-sphere $b_S$. But the cutting move applied to the Kirby loop simply removes it, thus showing that the diagram commutes.

\underline{R6 (case $k=0$)}. When we swap $S : S^0 \times D^2 \hookrightarrow \Sigma$ for its conjugate $\overline{S} : S^0 \times D^2 \hookrightarrow \Sigma$, the net effect is that the edges in the Kirby loop in $Z_\text{SN}(\Sigma(\overline{S}))$ will have the opposite orientation to the edges in the Kirby loop in $Z_\text{SN}(\Sigma(S))$. But this represents the same string net, since $d_i = d_{i^*}$ by sphericality (see Remark \ref{kirby_well_defined}). 

\underline{R6 (case $k=1$)}. Given a framed $1$-sphere $S : S^1 \times D^1 \hookrightarrow \Sigma$, consider the construction in Definition \ref{defn_of_string_modifications} of $Z_\text{SN}(e_S)$ carefully.  We are instructed to pull back the string-net inside the image of $D$ to the plane, and then cut it, 
\be
  \ig{d6.pdf} \quad \mapsto \quad  \sum_\alpha \ig{d8.pdf} \, , \label{normal_cut}
\ee
and then sew it back in. Here $e_\alpha$ is a basis for $\Hom(1, V_1 \otimes \cdots \otimes V_m$ and $e^\alpha$ is the dual basis for $\Hom(1, V_m^* \otimes \cdots \otimes V_1^*)$ according to the pairing \eqref{pairing}.

On the other hand, when we compute $Z_\text{SN}(e_{\overline{S}})$, we must effectively rotate the LHS of \eqref{normal_cut} by 180 degrees, then cut it, 
\be
 \ig{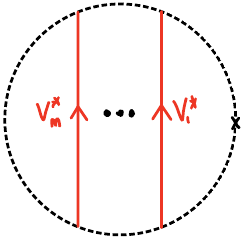} \quad \mapsto \quad \sum_\alpha \ig{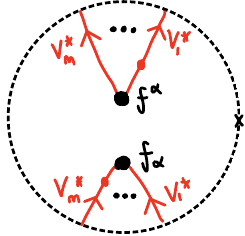} \, ,
\ee
and then rotate it back by 180 degrees and sew it back in. Here $f_\alpha$ is a basis for $\Hom(1, V_m^* \otimes \cdots \otimes V_1^*)$, and $f^\alpha$ is the dual basis for $\Hom(1, V_1 \otimes \cdots \otimes V_m)$, also according to the pairing \eqref{pairing}. In other words, we need to check whether
\[
 \sum_\alpha \ig{d8.pdf} = \sum_\alpha \ig{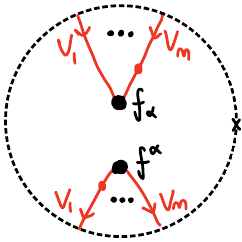} \, .
\]
But this can be proved by setting $f_\alpha := e^\alpha$, for then the symmetry property \eqref{symmetry_of_pairing} of the pairing (which was a consequence of pivotality) guarantees that $f^\alpha = e_\alpha$.
\end{proof}
We also have the following.
\begin{theorem2*} Given a spherical fusion category $C$, the string-net TQFT $Z_\text{SN}$ based on $C$ is naturally isomorphic to the Turaev-Viro TQFT $Z_\text{TV}$ based on $C$.
\end{theorem2*}
\begin{proof} This follows from the fact that, if we choose decompositions of each surface $\Sigma$ into pants, copants, cup and cap building blocks, and restrict Juh\'{a}sz' infinite list of surgery generating moves for $\ThreeBord$ listed in Definition \ref{juhasz_generators} to the finite (but sufficient) list given in the presentation of $\Bord_{123}^\text{or}$ from \cite{PaperIV} compatible with the chosen decompositions, then our assignments $e \mapsto Z_\text{SN}(e)$ precisely match those of Goosen \cite[Theorem 71]{goosen2018oriented} at the 23-level. These assignments were shown in \cite[Theorem 101]{goosen2018oriented} to give a 123 string-net TQFT $Z^\text{123}_\text{SN}$ which is equivalent to the 123 Turaev-Viro TQFT $Z^\text{123}_\text{TV}$. Restricting to the 23 sector gives the result.
\end{proof}

\end{definition}
\bibliographystyle{plain}
\bibliography{/Users/brucebartlett/Sync/Research/bibliography/references.bib}

\end{document}